\documentclass[reqno, 11pt]{amsproc}

\usepackage{amsmath, amssymb, amsthm}
\usepackage{graphicx}

\theoremstyle{plain}
\newtheorem{theorem}{Theorem}
\newtheorem{proposition}[theorem]{Proposition}
\newtheorem{lemma}[theorem]{Lemma}
\newtheorem{corollary}[theorem]{Corollary}

\theoremstyle{definition}

\begin{document}

\title{Non-periodic not everywhere
dense trajectories in triangular billiards}

\author[J.~Slipantschuk]{J.~Slipantschuk}
\address{%
J.~Slipantschuk\\
Department of Mathematics\\
University of Warwick\\
Coventry  CV4 7AL\\
UK.
}
\email{julia.slipantschuk@warwick.ac.uk}

\author[O.F.~Bandtlow]{O.F.~Bandtlow}
\address{%
O.F.~Bandtlow\\
School of Mathematical Sciences\\
Queen Mary University of London\\
London E3 4NS\\
UK.
}
\email{o.bandtlow@qmul.ac.uk}

\author[W.~Just]{W.~Just}
\address{%
W.~Just\\
Institut f\"ur Mathematik\\
Universit\"at Rostock\\
D-18057 Rostock\\
Germany. 
}
\email{wolfram.just@uni-rostock.de}

\date{June 13, 2024}


\maketitle

\begin{abstract}
Building on tools that have been successfully
used in the study of rational billiards,
such as induced maps and interval exchange transformations, we provide a construction of a one-parameter family of isosceles triangles exhibiting
non-periodic trajectories that are not everywhere dense.
This provides, by elementary means, 
a definitive negative answer to a long-standing
open question on
the density of non-periodic trajectories in triangular billiards.
\end{abstract}

\section{Introduction and results}\label{sec1}

Billiards, that is, the ballistic motion of a point particle in the plane with elastic collisions at
the boundary, are among the simplest mechanical systems producing intricate
dynamical features and thus serve as a paradigm in applied dynamical
systems theory \cite{ChMa:06}. The seemingly trivial case of billiards with piecewise straight
boundaries, known as polygonal billiards, offers surprisingly hard challenges \cite{Gutk_CHAOS12}.
When the inner angles of the polygon are rational multiples of $\pi$ the billiard dynamics
is dominated by a collection of conserved quantities and a rigorous and sophisticated
machinery for its treatment becomes available, see, for example, \cite{Gutk_JSP96,MaTa:02} for overviews.
Much less is known in the irrational case. A notable exception is the proof
of ergodicity of Lebesgue measure for a topologically large class of irrational
polygonal billiards \cite{KeMaSm_AM86}.
It is however not clear what this result means for numerical simulations of the billiard dynamics, as 
numerical studies of polygonal billiards are inconclusive.
Depending on the shape of the polygon, correlations in irrational billiards may or may not exhibit
decay \cite{ArCaGu_PRE97,CaPr_PRL99},
and even the ergodicity of Lebesgue measure
has been questioned \cite{WaCaPr_PRE14}. The relevance of symmetries has
been emphasised as an explanation for this conundrum \cite{ZaSlBaJu_PRE22}
and these predominantly numerical studies are underpinned by 
well established rigorous results on recurrence in polygonal billiards, 
see for instance 
\cite{Trou_RCD04, Trou_AIF05}.

In this article we shall be concerned with the simplest examples of polygonal billiards, namely those of triangular
shape. In particular we shall revisit a hypothesis formulated by Zemlyakov, see \cite{Galp_CMP83},
according to which trajectories are either periodic or cover the billiard table densely. While
\cite{Galp_CMP83} shows that this dichotomy does not hold in convex\footnote{For the simpler case of non-convex billiards, McMullen constructed an L-shaped example for which this dichotomy fails.} polygonal billiards with
more than three sides, the proof is flawed for triangular billiards, as pointed out
recently in \cite{Toka_CMP15}. Thus the existence of non-periodic and not everywhere dense
trajectories in triangular billiards remains an open problem, see also
\cite{DaDiRuLa_AG18} and references therein. We will fill this
gap by constructing trajectories of this type for a large set of symmetric
triangular billiards.
For this purpose, similarly to \cite{Galp_CMP83}, we reduce this problem to the properties of an induced one-dimensional map, a technique more commonly used in the case of rational billiards.
Leaving details of the
notation for later sections, we will prove the following.

\begin{theorem}\label{thrm1}
Consider a billiard map in the isosceles triangle determined by inner angles
$(\alpha,\alpha,\pi-2\alpha)$ with base angle $\alpha \in (\alpha_*,3\pi/10)$ for some $\alpha_*$ satisfying $\pi/4<\alpha_*<2\pi/7$.
Then there exist an angle $\phi_* \in (0, \pi)$ and an induced map on the base of the triangle $\{[k=1,\phi_*,x]: x\in [0,1]\}$ which is a rotation on the 
unit interval 
$x \mapsto x+\omega \mod 1$ with
$\omega =\cos(3\alpha)/(2 \cos(\alpha)\cos(4 \alpha))$.
\end{theorem}

As the rotation number $\omega$ is continuous and strictly mononotic for the given range of $\alpha$, 
this theorem implies that non-periodic 
not everywhere dense trajectories exist
in the billiard dynamics of the isosceles triangle for all but a countable subset of $\alpha \in (\alpha_*,3\pi/10)$,
providing a negative answer to the hypothesis by Zemlyakov. More precisely, we have the following corollary.

\begin{corollary}\label{coro1}
For all $\alpha \in (\alpha_*,3\pi/10)$
with $\cos(5\alpha)/\cos(3 \alpha) \in \mathbb{R}\setminus\mathbb{Q}$, in particular for all $\alpha\in (\alpha_*,3\pi/10)$ with $\alpha/\pi$ algebraic and irrational, the billiard dynamics
in the isosceles triangle contains trajectories which are non-periodic and not
everywhere dense in the triangle.
\end{corollary}

The main idea of the proof can be gleaned from the Zemlyakov-Katok
unfolding of the billiard dynamics \cite{ZeKa_MNAS75}.
Unfolding
the dynamics in a particular direction determined by an orbit starting and ending at a vertex (known as a generalised diagonal),
it can be seen that
the dynamics takes place in two recurrent cylinders, see Figure~\ref{fig1}.
This occurs, for instance, for an
isosceles triangle with base angle $\alpha=\pi \sqrt{3}/6$. For the geometrically minded
reader we summarise the essence of the proof. The existence of the required cylinders
can be verified by an explicit coordinatisation of the vertices in Figure \ref{fig1}.
Moreover, symmetry considerations show that this cylinder configuration
persists for all values $\alpha$ in a certain neighbourhood of $\pi\sqrt{3}/6$.
As a result it is then possible to introduce an induced map (on the base of the triangle),
which turns out to be a rotation with rotation number
varying continuously with $\alpha$, in particular taking
irrational values for a full-measure subset of the admissible
range of $\alpha$ values.
The construction also reveals
that the cylinders do not cover the whole interior of the
triangle, thus yielding non-periodic and not everywhere dense trajectories,
together forming a non-trivial flat strip in the sense of \cite{BoTr_FM11}.
Most of this article is devoted to making this
argument rigorous and completely explicit by an algebraic approach.
The idea for the geometric construction depicted in Figure \ref{fig1}
has been reported in \cite{ZaSlBaJu_PD23} where anomalous dynamics and 
recurrence in triangular billiards has been studied by a combination of
numerical computations and analytic arguments. The exposition contained 
in that reference provides compelling numerical evidence that the dynamics
in a particular direction is governed by an irrational rotation map, but
a rigorous proof for this observation has not been provided so far.

\begin{figure}[!h]
\centering
\includegraphics[width=0.7\textwidth]{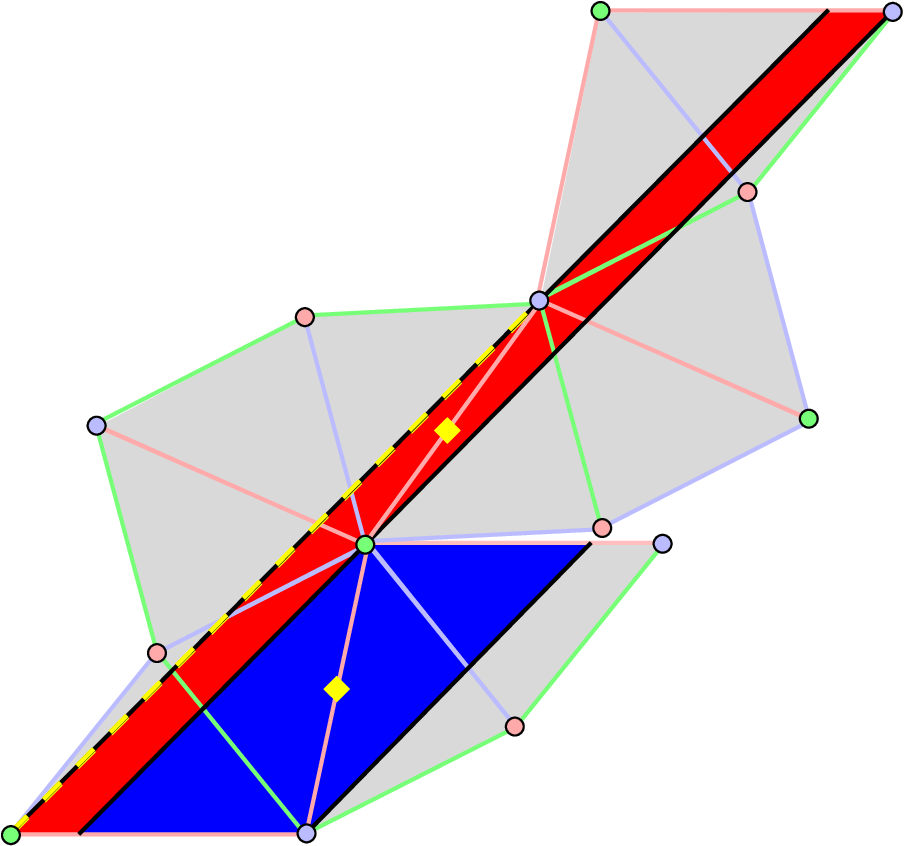}
\caption{Diagrammatic view of the Zemlyakov-Katok unfolding of
the billiard dynamics in an isosceles triangle with a 10- (red) and a 4-recurrent
cylinder (blue), respectively. Each cylinder is point symmetric with respect to
the midpoint of the central base (yellow diamond).
The unfolded generalised diagonal is shown in dashed yellow.
The three sides of the triangle and the vertices
have been coloured (base: light red, right leg: light green, left leg: light blue). 
\label{fig1}}
\end{figure}

We note that the proof of  Corollary \ref{coro1}
implies that the constructed billiard trajectories
never visit a certain neighbourhood around a tip of the isosceles triangle. This neighbourhood
can be replaced by a polygonal one, forming a convex $n$-gon for any $n\geq 4$, thus providing
an alternative, accessible, and elementary 
proof of Theorem 1 in \cite{Galp_CMP83} on the existence of non-periodic and not
everywhere dense billiard trajectories in convex $n$-gons.

We would like to mention in passing, that the complementary question of characterising billiards satisfying
the so-called ``Veech dichotomy'', that is those with the property that each direction is either completely periodic or uniquely ergodic,
has received significant attention in the literature in the case of rational billiards; see for example \cite{Gutk_ETDS84, Veec_IM89, KeSm_CMV00, Mcmu_IM06}.

This article is organised as follows. After fixing notation in Section \ref{sec2}, the existence of the
generalised diagonal will be established by Lemma \ref{lem6} in Section \ref{subsec2}.
We then turn to the existence and the properties of the two recurrent cylinders in
Proposition \ref{prop1} in Section \ref{subsec3} and Proposition \ref{prop2} in
Section \ref{subsec4}, respectively. The symmetry of the triangle will be instrumental
in setting up these cylinders and Lemma \ref{lem9} of Section \ref{subsec3}
summarises the main impact of the symmetry.
The proof of the main results follows standard arguments and will be presented
in Section \ref{subsec5}.

Our construction works for a limited range of base angles $\alpha\in(\alpha_*,3\pi/10)$.
For base angles outside this range the particular generalised diagonal or the recurrent
cylinders of period ten and four cease to exist. Nevertheless we suspect that the main
conclusion of Corollary \ref{coro1}, the existence of non-periodic not everywhere dense
orbits holds for almost all isosceles triangles. Analogous constructions
can be performed for other angles, but a more systematic approach would 
be needed to cover the general case. For trivial reasons
analogous statements hold in right-angled triangles.

\section{Notation and billiard map}\label{sec2}

Consider a triangle with positively oriented boundary. The sides
are labelled by a  cyclic index $k=1,2,3$. We denote by $s_k$ the
length of side $k$. The side with label $k=1$ is called the base.
We chose units of length such that $s_1=1$. Denote by  $\gamma_2$ and
$\gamma_3$ the left and right inner angle on the base, respectively. The angle opposite
to the base is denoted by $\gamma_1$. We shall focus exclusively on the case
of isosceles triangles, that is, $\gamma_2=\gamma_3=\alpha$. It readily follows that
$s_2=s_3=1/(2 \cos(\alpha))$.

The ballistic motion of a point particle with elastic bounces on the sides
of the triangle traces out a planar curve consisting of straight line segments. We call
this curve the \emph{trajectory}. We denote by $x_t^{[k]}$, $0<x_t^{[k]}<s_k$,
the location of the bounce of the particle (at discrete time $t$) at side $k$,
and by $\phi_t^{[k]}\in (0,\pi) $ the angle between the oriented side and the outgoing ray
of the trajectory. We call a move \emph{counter-clockwise (ccw)} if a bounce on side $k$ is
followed by a bounce on side $k+1$. Similarly we call a move \emph{clockwise (cw)}
if a bounce on side $k$ is followed by a bounce on side $k-1$. Subsequent
bounces are related by the \emph{billiard map}
\begin{equation}\label{eq:bm}
(x_{t}^{[k_{t}]}, \phi_{t}^{[k_{t}]} )\mapsto  (x_{t+1}^{[k_{t+1}]}, \phi_{t+1}^{[k_{t+1}]} )
\end{equation}
where
\begin{align}
\phi_{t+1}^{[k_{t+1}]} &=
\left\{
\begin{array}{lcl}
\pi-\phi_t^{[k_t]}  -\gamma_{k_t-1} & \mbox{ if } & k_{t+1}=k_t+1 \mbox{ (ccw)} \\
\pi-\phi_t^{[k_t]}  + \gamma_{k_t+1} & \mbox{ if } & k_{t+1}=k_t-1 \mbox{ (cw)}
\end{array} \right. \label{eq1}  \\
x_{t+1}^{[k_{t+1}]} &=
\left\{
\begin{array}{lcl}
(s_{k_t}-x_t^{[k_t]}) \sin(\phi_t^{[k_t]})/\sin(\phi_{t+1}^{[k_{t+1}]}) &
\mbox{ if } & k_{t+1}=k_t+1 \mbox{ (ccw)}\\
s_{k_{t+1}}-x_t^{[k_t]} \sin(\phi_t^{[k_t]})/\sin(\phi_{t+1}^{[k_{t+1}]}) & \mbox{ if } &
k_{t+1}=k_t-1 \mbox{ (cw)}
\end{array} \right. \label{eq2}
\end{align}
As it will be useful to keep track of the sequence of bouncing sides, we use a slightly non-standard notation
and call an \emph{orbit} a finite or infinite sequence of triplets $([k_t,\phi_t^{[k_t]},x_t^{[k_t]}])_{t \in I}$ which
obeys the billiard map (\ref{eq:bm}).
Each orbit corresponds to a
trajectory in the plane, and vice versa. We call a point $[k,\phi^{[k]},x^{[k]}]$ 
\emph{singular}, if it corresponds to one of the corners of the triangle, that is, if $x^{[k]}=0$ or
$x^{[k]}=s_k$. We call an orbit \emph{regular} if all its points are non-singular. An orbit which starts
or ends at a singular point will be referred to as a \emph{singular orbit}, while a billiard orbit which starts and ends
at a singular point is called a \emph{generalised diagonal}\footnote{When viewing
the billiard flow as a flow on a translation surface,
this is also referred to as a \emph{saddle connection}.}.

For a fixed side $k$ of the triangle and a fixed angle $\phi$, we will refer to
the family of parallel trajectory segments reflecting from $k$ at angle $\phi$ and returning to $k$ with the same
angle $\phi$ after a fixed sequence of bouncing sides as a \emph{recurrent cylinder}.

\section{Proof of results}\label{sec3}

Our proof consists of several steps. We will first establish the existence of a suitable
induction angle, such that an orbit emanating from the left endpoint
of the base at this angle forms a generalised diagonal with a certain length-5 sequence
of bouncing sides. We will then show
that \emph{every} orbit emanating from the base at this angle returns to the base with the same angle after a fixed number of bounces (either 10 or 4, depending on the initial point). The two corresponding sets of billiard trajectories will form two recurrent
cylinders in the plane, crucially bounded away by a positive distance from one of the
triangle's vertices.
This construction will yield an induced map, forming an interval exchange transformation over two subintervals of the triangle base. The rotation number of this
interval exchange transformation will depend continuously on the angle of the isosceles triangle, implying an irrational rotation and hence dense trajectories in the union of the
recurrent cylinders for a large set of angles of the triangle.

\subsection{Induction angle}\label{subsec1}
We begin by proving several lemmas, which will be used to establish that for
a suitable range of values of $\alpha$ there exists an angle $\phi_*$ (depending on $\alpha$), such that
the orbit emanating from the left endpoint of the base at angle $\phi_*$ forms a generalised diagonal.

\begin{lemma}\label{lem1}
For $\alpha \in (\pi/4, 3 \pi/10)$ the equation $g(\alpha)=\sin(7 \alpha)-\sin(3 \alpha)+
\sin(\alpha)=0$ has a unique solution $\alpha_* \in (\pi/4, 2 \pi/7)$.
\end{lemma}

\begin{proof}
We have that $g(\pi/4)=\sin(7\pi/4)<0$ and $g(3 \pi/10)=\sin(3\pi/10)>0$.
Since $7 \alpha \in (7 \pi/4, 21\pi/10)$, $3\alpha\in(3 \pi/4,9\pi/10)$ and
$\alpha\in (\pi/4, 3 \pi/10)$ it follows that
\begin{displaymath}
g'(\alpha)=7 \cos(7 \alpha)-3\cos(3 \alpha)+ \cos(\alpha)>0 \, .
\end{displaymath}
Existence and uniqueness of $\alpha_*\in (\pi/4, 3 \pi/10)$ now follow
from a variant of the intermediate value theorem.
For the remaining assertion observe that
$g(2\pi/7)=\sin(2\pi/7)-\sin(6\pi/7)=\sin(2\pi/7)-\sin(\pi/7)>0$.
\end{proof}

The angle $\alpha_*$ established by Lemma \ref{lem1} is the value of the base angle
where the geometry shown in Figure~\ref{fig1} starts to break down, since the generalised diagonal
hits the top vertex of the triangle, as we will show shortly.

\begin{lemma}\label{lem2}
Let $\alpha\in [\pi/4,3 \pi/10]$. The equation
\begin{equation}
g(\alpha,\phi)=\sin(6 \alpha+\phi)-\sin(2 \alpha+\phi)+\sin(\phi)=0 \label{eq3}
\end{equation}
has a unique solution $\phi=\phi_*(\alpha)$ in $(0,\pi)$.
\end{lemma}

\begin{proof}
We have $g(\alpha,0)=2 \sin(2 \alpha) \cos(4 \alpha)<0$.
Existence and uniqueness of the solution in $(0,\pi)$ follow from the
observation that $g(\alpha,\phi)$ is a Fourier polynomial in $\phi$
containing only the two first order terms.
\end{proof}

For the range $\alpha\in(\alpha_*,3\pi/10)$ of base angles, Lemma
\ref{lem2} defines a direction $\phi_*$ of the directional billiard
flow which determines the unfolding shown in Figure \ref{fig1}. This flow
will be instrumental in proving our main result.

\begin{lemma}\label{lem3}
Let $\alpha\in (\alpha_*, 3 \pi/10)$. The solution to (\ref{eq3}) obeys
\begin{equation}\label{eq4}
0<\phi_*<\alpha, \quad \alpha+\phi_* > \frac{\pi}{2},
\quad 3 \alpha+\phi_*>\pi, \quad 6 \alpha+\phi_*<2 \pi
< 7 \alpha+ \phi_*  \, .
\end{equation}
\end{lemma}

\begin{proof}
Using the substitution
\begin{displaymath}
\alpha= 3 \pi/10+x, \quad \phi_*= \pi/5+y
\end{displaymath}
with $-\pi/20<x<0$ (equivalent to $\alpha\in (\pi/4,3\pi/10)$), equation (\ref{eq3}) reads
\begin{equation}
\bar{g}(x,y)=\sin(6x + y) + 2 \sin(x+y) \sin(3\pi/10+x)=0\,. \label{eq5}
\end{equation}
We have that
\begin{displaymath}
\bar{g}(x,-x)=\sin(5 x)<0, \quad \bar{g}(x,-6x)=2 \sin(-5x) \sin(3\pi/10+x) >0\,.
\end{displaymath}
It follows that $-x<y<-6x$ with $-\pi/20 < x < 0$, and therefore
\begin{align*}
\alpha+\phi_* &=  3 \pi/10+\pi/5 + x +y > \pi/2,\\
6 \alpha + \phi_* &=  9 \pi/5 + \pi/5 + 6x +y < 2 \pi,\\
3 \alpha+\phi_* &=  9 \pi/10 +3 x + \pi/5 +y = \pi + x+y+\pi/10+2x > \pi,\\
7 \alpha + \phi_* &=  6 \alpha +  \alpha + \phi_* >
3 \pi/2 + \pi/2 \, .
\end{align*}
Treating $y$ in (\ref{eq5}) as a function of $x$, implicit differentiation yields
\begin{align*}
0 =&  \frac{dy}{dx}\left(\cos(6x+y)+2 \cos(x+y)\sin(3\pi/10+x)\right)\\
&  +6 \cos(6x+y)+2 \cos(x+y) \sin(3\pi/10+x)+ 2 \sin(x+y) \cos(3\pi/10+x)\,.
\end{align*}
Since $-3\pi/10<6x+y<0$, $0<x+y<3\pi/10$, and $\pi/4<3\pi/10+x<3\pi/10$, all
trigonometric terms are positive and $dy/dx<0$. Hence the solution
$\phi_*(\alpha)$ is a strictly monotonic decreasing function for $\alpha\in
(\pi/4,3\pi/10)$. Since Lemma \ref{lem1} and \ref{lem2} imply $\phi_*(\alpha_*)=\alpha_*$
the final assertion follows.
\end{proof}
For the remainder of the paper we will refer to the value obtained in Lemma~\ref{lem1} as $\alpha_*$,
and for $\alpha \in (\alpha_*, 3\pi/10)$ we will write $\phi_* = \phi_*(\alpha)$,
omitting the dependence on the angle $\alpha$ where there is no risk of ambiguity.

\subsection{Generalised diagonal}\label{subsec2}
Next, we proceed to show the existence of a generalised diagonal
starting from the left endpoint of the base at angle $\phi_*$.
For this, we will ascertain that the formal recurrence equations (\ref{eq1}) and (\ref{eq2})
are satisfied by a given sequence of bouncing sides, angles, and spatial coordinates,
which therefore form a valid (that is, realisable) orbit.
We define the sequence of bouncing sides
\begin{equation}\label{eq6}
(m_t)_{0\leq t \leq 5}=(1,2,3,1,3,1)
\end{equation}
and introduce the sequence of angles
\begin{equation}
\begin{array}{lll}
\psi_0=\phi_*, \quad &
\psi_1=\pi-\alpha-\phi_*, \quad &
\psi_2=-\pi+3\alpha+\phi_*,  \\
\psi_3=2\pi-4 \alpha-\phi_*, \quad &
\psi_4=-\pi+5\alpha+\phi_*, \quad &
\psi_5=2\pi-6\alpha-\phi_*\, .
\end{array}
\label{eq7}
\end{equation}
It is straightforward to check that the angles (\ref{eq7})
together with (\ref{eq6}) satisfy the recurrence (\ref{eq1}). Furthermore, Lemma
\ref{lem3} yields the following result.

\begin{lemma}\label{lem4}
Let $\alpha\in (\alpha_*,3\pi/10)$.
The angles defined by (\ref{eq6}) and (\ref{eq7}) obey $0<\psi_t<\pi$,
$0\leq t \leq 5$.
\end{lemma}

Define, for $\delta\in \mathbb{R}$, the spatial coordinates
\begin{align}
\xi_0(\delta)&= \delta, \nonumber \\
\xi_1(\delta)&= (s_1-\delta) \frac{\sin(\psi_0)}{\sin(\psi_1)}, \nonumber \\
\xi_2(\delta)&= s_2 \frac{\sin(\psi_1)}{\sin(\psi_2)}
-(s_1-\delta)\frac{\sin(\psi_0)}{\sin(\psi_2)}, \nonumber \\
\xi_3(\delta)&=  s_3 \frac{\sin(\psi_2)}{\sin(\psi_3)}-
s_2 \frac{\sin(\psi_1)}{\sin(\psi_3)}
+(s_1-\delta)\frac{\sin(\psi_0)}{\sin(\psi_3)}, \nonumber \\
\xi_4(\delta)&= s_3-s_3 \frac{\sin(\psi_2)}{\sin(\psi_4)}+
s_2 \frac{\sin(\psi_1)}{\sin(\psi_4)}
-(s_1-\delta)\frac{\sin(\psi_0)}{\sin(\psi_4)}, \nonumber \\
\xi_5(\delta)&= s_3 \frac{\sin(\psi_2)}{\sin(\psi_5)}-
s_2 \frac{\sin(\psi_1)}{\sin(\psi_5)}
+(s_1-\delta)\frac{\sin(\psi_0)}{\sin(\psi_5)} \, . \label{eq8}
\end{align}
It is again straightforward to check that the expressions in (\ref{eq8}) together with
(\ref{eq6}) and (\ref{eq7}) obey the formal recurrence in (\ref{eq2}).

\begin{lemma}\label{lem5}
Let $\alpha\in (\alpha_*,3\pi/10)$. The coordinates defined in (\ref{eq6}), (\ref{eq7}),
and (\ref{eq8}) satisfy $\xi_0(0)=0$, $\xi_5(0)=1$, and $0<\xi_t(0)<s_{k_t}$, $1\leq t \leq 4$.
\end{lemma}

\begin{proof}
The initial coordinate $\xi_0(0)=0$ is obvious.
From Lemma \ref{lem2} we have
\begin{align*}
0 &=  2 \cos(\alpha) \left(\sin(6 \alpha+ \phi_*)-\sin(2 \alpha+\phi_*)
+\sin(\phi_*)\right) \\
&=  2 \cos(\alpha) \sin(6 \alpha+ \phi_*)
-\sin(3 \alpha + \phi_*)-\sin(\alpha+\phi_*) + 2 \cos(\alpha) \sin(\phi_*)\,.
\end{align*}
With (\ref{eq7}) and $1=s_1=2 \cos(\alpha) s_{2/3}$ this reads
\begin{equation}
\sin(\psi_5) = s_3 \sin(\psi_2)-s_2 \sin(\psi_1)+
s_1 \sin(\psi_0) \, , \label{eq9}
\end{equation}
which implies $\xi_5(0)=1$.

Using (\ref{eq9}) and Lemma \ref{lem4} we have
\begin{displaymath}
\xi_4(0)=s_3-\sin(\psi_5)/\sin(\psi_4) <s_3 \, .
\end{displaymath}
Furthermore, by Lemma \ref{lem3} we have
\begin{align*}
\sin(\psi_4)-2 \cos(\alpha)\sin(\psi_5) =
\sin(7 \alpha + \phi_*) >0 \, ,
\end{align*}
which implies $\xi_4(0)>0$.

Again using (\ref{eq9}) and Lemma \ref{lem4} we have
\begin{displaymath}
\xi_3(0)=\sin(\psi_5)/\sin(\psi_3) > 0 \, ,
\end{displaymath}
and by Lemma \ref{lem3} and $2\alpha>\pi/2$ we obtain
\begin{displaymath}
\sin(\psi_3)-\sin(\psi_5)=2 \sin(\alpha) \cos(5 \alpha+\phi_*)>0 \, ,
\end{displaymath}
which implies $\xi_3(0)<1=s_1$.

Lemma \ref{lem3} implies
\begin{displaymath}
0<\sin(\alpha-\phi_*)=\sin(\alpha+\phi_*)-2 \cos(\alpha)\sin(\phi_*)
\end{displaymath}
so that, using the abbreviations (\ref{eq7}) we have
\begin{equation}\label{eq10}
\sin(\psi_0)<s_2 \sin(\psi_1) \, .
\end{equation}
Hence $\xi_2(0)>0$. Furthermore, (\ref{eq9}) and Lemma \ref{lem4} yield
\begin{displaymath}
0<s_3 -s_2 \frac{\sin(\psi_1)}{\sin(\psi_2)}+s_1 \frac{\sin(\psi_0)}{\sin(\psi_2)} \, ,
\end{displaymath}
which is equivalent to $\xi_2(0)<s_3$.

Finally, $\xi_1(0)>0$ is obvious, and $\xi_1(0)<s_2$ follows from (\ref{eq10}).
\end{proof}

Lemma \ref{lem4} and \ref{lem5} now yield the following conclusion.

\begin{lemma}\label{lem6}
Let $\alpha\in (\alpha_*,3\pi/10)$. With the definitions (\ref{eq6}), (\ref{eq7}),
and (\ref{eq8}) the sequence $([m_t,\psi_t,\xi_t(0)])_{0\leq t \leq 5}$ defines
a generalised diagonal.
\end{lemma}

Lemma \ref{lem6} establishes the generalised diagonal, shown in Figure \ref{fig1} as a
dashed yellow line, by purely algebraic means. The generalised diagonal determines
the direction $\phi_*$ of the unfolding. If the base angle of the triangle,
$\alpha$, drops below the critical value $\alpha_*$ this connection ceases to exist.
At $\alpha=\alpha_*$ the generalised diagonal hits the top vertex of the first
triangle in the unfolding, as can be gleaned from Figure \ref{fig1}.
This geometric condition poses the major constraint on the existence of the generalised diagonal.

\subsection{Recurrent cylinder of length ten}\label{subsec3}

In this section we will establish the existence of a point $x_D$ on the base
of the triangle, such that all points in $(0, x_D) \times \{\phi_*\}$ share the same length-10 sequence of bouncing sides.
Using a symmetry of the triangle, this sequence will be shown to consist of the length-5 sequence (\ref{eq6}),
followed by a `mirrored' variant of the same sequence,
in a sense made precise below. Moreover we will observe
that the image of $(0, x_D) \times \{\phi_*\}$ under the 10th iteration
of the billiard map is $(1-x_D, 1) \times \{\phi_*\}$. The orbit of the point $(x_D, \phi_*)$ itself
will be singular, giving rise to a discontinuity of the induced map on the base.
We begin by defining
\begin{equation}\label{eq11}
x_D=1-\frac{\sin(2 \alpha+\phi_*)}{\sin(\phi_*)}\, .
\end{equation}

\begin{lemma}\label{lem6a}
Let $\alpha \in (\alpha_*, 3\pi/10)$. The quantity defined by (\ref{eq11}) obeys
$x_D \in (0,1)$ and
\begin{displaymath}
x_D=\frac{\sin(\psi_5)}{\sin(\psi_0)}=1-\frac{\cos(3 \alpha)}{2 \cos(\alpha) \cos(4 \alpha)} \, .
\end{displaymath}
\end{lemma}

\begin{proof}
Lemma \ref{lem3} implies $2 \alpha+\phi_*<\pi$ so that $x_D<1$. Furthermore
\begin{displaymath}
\sin(2 \alpha+\phi_*)-\sin(\phi_*)=2 \sin(\alpha)\cos(\alpha+\phi_*)<0
\end{displaymath}
so that $x_D>0$.
Using Lemma \ref{lem2} we have
\begin{displaymath}
x_D=\frac{-\sin(6 \alpha+\phi_*)}{\sin(\phi_*)}=\frac{\sin(\psi_5)}{\sin(\psi_0)} \, .
\end{displaymath}
Furthermore, (\ref{eq3}) yields
\begin{displaymath}
\left(\cos(6 \alpha)-\cos(2 \alpha)+1\right)  \sin(\phi_*)
+ \left(\sin(6 \alpha)-\sin(2 \alpha)\right) \cos(\phi_*)=0
\end{displaymath}
so that
\begin{align*}
\frac{\sin(2 \alpha + \phi_*)}{\sin(\phi_*)} &=
-\sin(2 \alpha)\frac{\cos(6 \alpha)-\cos(2 \alpha)+1}{\sin(6 \alpha)-
\sin(2 \alpha)}+\cos(2 \alpha)\\
&=  \frac{\sin(4 \alpha)-\sin(2 \alpha)}
{\sin(6 \alpha)-\sin(2 \alpha)}=\frac{\cos(3 \alpha)}{2 \cos(\alpha)
\cos(4 \alpha)}  \, .
\end{align*}
\end{proof}

\begin{lemma}\label{lem7}
Let $\alpha\in (\alpha_*,3\pi/10)$. The coordinates defined in
(\ref{eq6}), (\ref{eq7}),
and (\ref{eq8}) obey $0<\xi_t(x_D)<s_{m_t}$, $t=0,1$,
$\xi_2(x_D)=\xi_4(x_D)=s_3$, and $\xi_3(x_D)=\xi_5(x_D)=0$.
\end{lemma}

\begin{proof}
Since $\xi_0(x_D)=x_D$, Lemma \ref{lem6a} yields the assertion for $t=0$.

Using (\ref{eq9}) and Lemma \ref{lem6a} we have
\begin{displaymath}
\xi_5(x_D)=1-x_D \frac{\sin(\psi_0)}{\sin(\psi_5)} =0 \, .
\end{displaymath}
The assertions $\xi_3(x_D)=0$ and $\xi_2(x_D)= \xi_4(x_D)=s_3$, follow
from the equalities $\xi_3(\delta)=\xi_5(\delta) \sin(\psi_5)/\sin(\psi_3)$,
$\xi_2(\delta)=s_3-\xi_5(\delta) \sin(\psi_5)/\sin(\psi_2)$,
and $\xi_4(\delta)=s_3-\xi_5(\delta)\sin(\psi_5)/\sin(\psi_4)$.
By Lemma \ref{lem6a} we have $\sin(\psi_0)>\sin(\psi_5)$, which implies
\begin{displaymath}
\xi_1(x_D)=\frac{\sin(\psi_0)}{\sin(\psi_1)}-\frac{\sin(\psi_5)}{\sin(\psi_1)}>0 \, .
\end{displaymath}
Finally, using (\ref{eq10}) we obtain
\begin{displaymath}
\xi_1(x_D)=\frac{\sin(\psi_0)}{\sin(\psi_1)}-\frac{\sin(\psi_5)}{\sin(\psi_1)}
<   \frac{\sin(\psi_0)}{\sin(\psi_1)} < s_2 \, .
\end{displaymath}
\end{proof}

Since the angles and spatial coordinates defined in (\ref{eq6}), (\ref{eq7}), and (\ref{eq8}) obey the formal recurrence scheme
determined by the billiard map (\ref{eq:bm}),
Lemmas \ref{lem4} and \ref{lem7} yield the following result.

\begin{lemma}\label{lem8}
Let $\alpha\in (\alpha_*,3\pi/10)$. Then the following holds. Given any $\delta \in (0,x_D)$,
the sequence $([m_t,\psi_t,\xi_t(\delta)])_{0\leq t \leq 5}$
with components defined by (\ref{eq6}), (\ref{eq7}), and (\ref{eq8})
constitutes a regular orbit of the billiard map (\ref{eq:bm}).
\end{lemma}

The symmetry of the triangle has implications for the structure of orbits.
Reflecting an orbit at the symmetry axis of the triangle yields again an orbit.
In formal terms, this type of reflection is expressed as
$[k,\phi^{[k]},x^{[k]}] \mapsto [\bar{k},\pi-\phi^{[k]},s_k-x^{[k]}]$ where the
adjoint index $\bar{k}$ is given by $\bar{1}=1$, $\bar{2}=3$, $\bar{3}=2$.
Similarly, reversing the motion gives again an orbit. In formal terms,
the corresponding transformation reads $[k,\phi^{[k]},x^{[k]}] \mapsto [k,\pi-\phi^{[k]},x^{[k]}]$.
Combining both operations maps an orbit to another orbit.

\begin{lemma}\label{lem9}
If $([k_t,\phi_t^{[k_t]}, x_t^{[k_t]}])_{0\leq t \leq T}$ denotes a
finite regular orbit in a symmetric triangular billiard then $([\ell_t,\varphi_t^{[\ell_t]},
z_t^{[\ell_t]}])_{0\leq t \leq T}$ gives a finite regular orbit of the same length
where $\ell_t=\bar{k}_{T-t}$, $\varphi_t^{[\ell_t]}=\phi_{T-t}^{[k_{T-t}]}$ and
$z_t^{[\ell_t]}=s_{k_{T-t}} - x_{T-t}^{[k_{T-t}]}$. Here $\bar{k}$ denotes the adjoint
index defined by $\bar{1}=1$, $\bar{2}=3$, $\bar{3}=2$.
\end{lemma}

\begin{proof}
We first note the identity  $\overline{k\pm 1}=\bar{k}\mp 1$.
The symmetry of the triangle is equivalent to $\gamma_k=\gamma_{\bar{k}}$ and
$s_k=s_{\bar{k}}$. We consider a fixed time $t$, $0\leq t <T$.

Case A: Assume that the move $T-t-1 \rightarrow T-t$ in the original orbit is
counter-clockwise, that is, $k_{T-t}=k_{T-t-1}+1$. Then $\ell_t=\bar{k}_{T-t}=
\bar{k}_{T-t-1}-1=\ell_{t+1}-1$ (that is, the move $t\rightarrow t+1$ in the image orbit
is counter-clockwise as well).

Equation (\ref{eq1}) tells us that for the original angles we have
\begin{displaymath}
\phi_{T-t}^{[k_{T-t}]}= \pi - \phi_{T-t-1}^{[k_{T-t-1}]}-\gamma_{k_{T-t-1}-1} \, .
\end{displaymath}
Observing that
\begin{displaymath}
\gamma_{k_{T-t-1}-1}= \gamma_{\overline{k_{T-t-1}-1}}= \gamma_{\bar{k}_{T-t-1}+1}
=\gamma_{\ell_{t+1}+1}=\gamma_{\ell_t-1}\, ,
\end{displaymath}
we have
\begin{displaymath}
\varphi_t^{[\ell_t]}= \pi-\varphi_{t+1}^{[\ell_{t+1}]}-\gamma_{\ell_t-1}
\end{displaymath}
which is the angle dynamics of the billiard map for the image orbit.

Similarly, (\ref{eq2}) implies for the spatial coordinates of the original orbit that
\begin{displaymath}
x_{T-t}^{[k_{T-t}]}=\left(s_{k_{T-t-1}}-x_{T-t-1}^{[k_{T-t-1}]}\right)
\frac{\sin(\phi_{T-t-1}^{[k_{T-t-1}]})}{\sin(\phi_{T-t}^{[k_{T-t}]}) }
\end{displaymath}
so that
\begin{displaymath}
s_{k_{T-t}}-z_t^{[\ell_t]}=z_{t+1}^{[\ell_{t+1}]}
\frac{\sin(\varphi_{t+1}^{[\ell_{t+1}]})}{\sin(\varphi_t^{[\ell_t]})} \, .
\end{displaymath}
Recalling that $s_{k_{T-t}}=s_{\bar{k}_{T-t}}=s_{\ell_t}$ we obtain the
position dynamics of the billiard map for the image orbit.

Case B: The proof in case the move $T-t-1 \rightarrow T-t$ in the original orbit is
clockwise, that is, $k_{T-t}=k_{T-t-1}-1$, is similar.
\end{proof}

The symmetry allows us to extend the regular orbit derived in Lemma \ref{lem8}
to a recurrent orbit with $\phi_0^{[k_0]}=\phi_T^{[k_T]}$.

\begin{proposition}\label{prop1}
Let $\alpha\in (\alpha_*,3\pi/10)$. For any $\delta \in (0,x_D)$
there exists a recurrent regular orbit of length 10
given by $([k_t, \phi_t^{[k_t]}, x_t^{[k_t]}])_{0\leq t \leq 10}$
with initial condition $[k_0,\phi_0^{[k_0]},x_0^{[k_0]}]=[1,\phi_*,\delta]$ and
endpoint $[k_{10},\phi_{10}^{[k_{10}]},x_{10}^{[k_{10}]}]=[1,\phi_*,\delta+1-x_D]$.
The explicit expression for the orbit is given by $k_t=m_t$, $\phi_t^{[k_t]}=\psi_t$, and
$x_t^{[k_t]}=\xi_t(\delta)$ for $0\leq t \leq 5$ and $k_t=\bar{m}_{10-t}$,
$\phi_t^{[k_t]}=\psi_{10-t}$, and $x_t^{[k_t]}=s_{\bar{m}_{10-t}}-\xi_{10-t}(x_D-\delta)$
for $6\leq t \leq 10$.
\end{proposition}

\begin{proof}
Let $\delta \in (0,x_D)$. Lemma \ref{lem8} provides us with the regular
orbit of length 5, $([m_t,\psi_t,\xi_t(\delta)])_{0\leq t \leq 5}$, with
initial condition $[1,\phi_*,\delta]$ and endpoint $[1,\psi_5,\xi_5(\delta)]$.
Replacing $\delta$ by $x_D-\delta$, Lemma \ref{lem8} yields the regular length-5 orbit given by
$([m_t,\psi_t,\xi_t(x_D-\delta)])_{0\leq t \leq 5}$.
Applying Lemma \ref{lem9} we obtain the regular orbit
$([\bar{m}_{5-t},\psi_{5-t},s_{\bar{m}_{5-t}}-\xi_{5-t}(x_D-\delta)])_{0\leq t \leq 5}$
with initial condition $[1,\psi_5,s_1-\xi_5(x_D-\delta)]$ and endpoint
$[1,\phi_*,1-x_D+\delta]$. Recalling that (\ref{eq8}) and (\ref{eq9})
imply $\xi_5(\delta)=1-\delta \sin(\psi_0)/\sin(\psi_5)$ and using
Lemma \ref{lem6a}, we obtain $s_1-\xi_5(x_D-\delta)=\xi_5(\delta)$.
The assertions of the proposition follow by transitivity of orbits of the billiard map.
\end{proof}

Proposition \ref{prop1} establishes by purely algebraic means
the recurrent cylinder of length~$10$ depicted in Figure \ref{fig1}
as the red parallelogram. Due to the symmetry of the underlying triangle, the
entire structure shown in Figure \ref{fig1}
is point symmetric about the midpoint of this parallelogram.
Hence the top and bottom side of the parallelogram are guaranteed to be parallel
ensuring recurrence of the scattering angle of the billiard dynamics.
The same symmetry ensures that the right side of the parallelogram also
contains a generalised diagonal and links up with another periodic cylinder.
The width of the parallelogram, that is, the distance between the two long sides
is essentially given by (\ref{eq11}). Both sides approach each other
when increasing the base angle $\alpha$, and the width of the parallelogram
vanishes at the upper
critical angle $\alpha=3\pi/10$. Our algebra ensures that no further geometric
obstruction, that is, no further vertex, appears within the cylinder of period 10,
as can be gleaned from Figure \ref{fig1}.

\subsection{Recurrent cylinder of length four}\label{subsec4}

In an analogous way we can define a recurrent cylinder of length 4 for initial
conditions $x_0^{[1]}$ in $(x_D,1)$. For that purpose we define
\begin{equation}\label{eq12}
(\ell_t)_{0\leq t \leq 2} = (1,2,1),
\end{equation}
\begin{equation}\label{eq13}
\theta_0=\phi_*, \quad \theta_1= \pi-\alpha-\phi_*, \quad \theta_2= 2 \alpha+\phi_*,
\end{equation}
\begin{align}\label{eq14}
\eta_0(\delta) &=  \delta, \nonumber \\
\eta_1(\delta) &=  (s_1-\delta) \frac{\sin(\theta_0)}{\sin(\theta_1)},\nonumber \\
\eta_2(\delta) &= s_1-(s_1-\delta) \frac{\sin(\theta_0)}{\sin(\theta_2)}\,.
\end{align}
Formally, the angles and spatial coordinates defined in (\ref{eq12}), (\ref{eq13}), and (\ref{eq14}) obey the recursion scheme
of the billiard map (\ref{eq:bm}). In a similar vein to Lemma~\ref{lem8}
we have the following result.

\begin{lemma}\label{lem10}
Let $\alpha\in (\alpha_*,3\pi/10)$. Then the following holds. Given any $\delta \in (x_D,1)$, the sequence
$([\ell_t,\theta_t,\eta_t(\delta)])_{0\leq t \leq 2}$ with components defined by (\ref{eq12}), (\ref{eq13}), and (\ref{eq14})
constitutes a regular orbit of the billiard map (\ref{eq:bm}).
\end{lemma}

\begin{proof}
By Lemma \ref{lem3} we have that $0<\theta_t<\pi$ for $0\leq t \leq 2$.
Fixing $\delta \in (x_D,1)$, Lemma \ref{lem6a} yields $0<\eta_0(\delta)<1$.
We next observe that $\eta_2(1)=1$, while (\ref{eq11}) yields $\eta_2(x_D)=0$, and hence
$0<\eta_2(\delta)<1$.

Finally, we have $\eta_1(1)=0$; furthermore, since
\begin{displaymath}
2\cos(\alpha) \sin(2\alpha + \phi_*) - \sin(\alpha+\phi_*)
=\sin(3\alpha+\phi_*)<0
\end{displaymath}
by Lemma \ref{lem3}, we have
\begin{displaymath}
\eta_1(x_D)=\frac{\sin(2\alpha+\phi_*)}{\sin(\alpha+\phi_*)}<\frac{1}{2\cos(\alpha)},
\end{displaymath}
and so $0<\eta_1(x_D)<s_2$, which yields $0<\eta_1(\delta)<s_2$ for $\delta\in(x_D,1)$.
\end{proof}

Again employing the symmetry of the triangle yields the following result.

\begin{proposition}\label{prop2}
Let $\alpha\in (\alpha_*,3\pi/10)$. For any $\delta \in (x_D,1)$
there exists a recurrent regular orbit of length 4 given by
$([k_t, \phi_t^{[k_t]}, x_t^{[k_t]}])_{0\leq t \leq 4}$
with initial condition $[k_0,\phi_0^{[k_0]},x_0^{[k_0]}]=[1,\phi_*,\delta]$  and
endpoint $[k_{4},\phi_{4}^{[k_{4}]},x_{4}^{[k_{4}]}]=[1,\phi_*,\delta-x_D]$.
The explicit expression for the orbit is given by $k_t=\ell_t$, $\phi_t^{[k_t]}=\theta_t$, and
$x_t^{[k_t]}=\eta_t(\delta)$ for $0\leq t \leq 2$ and $k_t=\bar{\ell}_{4-t}$,
$\phi_t^{[k_t]}=\theta_{4-t}$, and $x_t^{[k_t]}=s_{\bar{\ell}_{4-t}}-\eta_{4-t}(1+x_D-\delta)$
for $3\leq t \leq 4$.
\end{proposition}

\begin{proof}
Let $\delta \in (x_D,1)$. Lemma \ref{lem10} provides us with the regular
orbit of length 2, $([\ell_t,\theta_t,\eta_t(\delta)])_{0\leq t \leq 2}$, with
initial condition $[1,\phi_*,\delta]$ and endpoint $[1,\theta_2,\eta_2(\delta)]$.
Replacing $\delta$ by $1+x_D-\delta$, Lemma \ref{lem10} yields the regular length-2 orbit given by
$([\ell_t,\theta_t,\eta_t(1+x_D-\delta)])_{0\leq t \leq 2}$.
Applying Lemma \ref{lem9} we obtain the regular orbit
$([\bar{\ell}_{s-t},\theta_{s-t},s_{\bar{\ell}_{2-t}}-\eta_{2-t}(1+x_D-\delta)])_{0\leq t \leq 2}$
with initial condition $[1,\theta_2,s_1-\eta_2(1+x_D-\delta)]$ and endpoint
$[1,\phi_*,\delta-x_D]$. Recalling that (\ref{eq11}), (\ref{eq13}) and (\ref{eq14})
imply $x_D=1-\sin(\theta_2)/\sin(\theta_0)$ and
$\eta_2(\delta)=1-(1-\delta) \sin(\theta_0)/\sin(\theta_2)$ we obtain
$s_1-\eta_2(1+x_D-\delta)=\eta_2(\delta)$.
The assertions of the proposition follow by transitivity of orbits of the billiard map.
\end{proof}

\subsection{Proof of the theorem and its corollary}\label{subsec5}

Propositions \ref{prop1} and \ref{prop2} constitute the proof of Theorem \ref{thrm1} with the expression for the rotation number $\omega=1-x_D$
following readily from Lemma \ref{lem6a}.
The proof of the corollary will be based on the following lemma which
summarises the findings in Lemma \ref{lem5} and \ref{lem10}.

\begin{lemma}\label{lem11}
Let $\alpha\in (\alpha_*,3\pi/10)$. There exists
$\varepsilon>0$ such that any infinite regular orbit
$([k_t,\phi_t^{[k_t]},x_t^{[k_t]})_{t\geq 0}$  of the billiard map with initial
condition $([k_0,\phi_0^{[k_0]},x_0^{[k_0]}]=[1,\phi_*,x_0^{[1]}]$ satisfies
the conditions $x_t^{[k_t]}\leq s_2-\varepsilon$ whenever $k_t=2$, and $x_t^{[k_t]}\geq  \varepsilon$ whenever
$k_t=3$.
\end{lemma}

\begin{proof}
The spatial coordinate $x_t^{[k_t]}$ does not take the values 0,1, or $x_D$ if $k_t=1$
as those are singularities or are mapped to singularities, see Lemma \ref{lem7}.
Furthermore,
by Proposition \ref{prop1} and \ref{prop2} the orbit is recurrent. Hence it is sufficient
to consider the 4- and 10-recurrent pieces of the orbit.

Consider $[k_0,\phi_0^{[k_0]} ,x_0^{[k_0]}]=[1,\phi_*,\delta]$ with $x_D<\delta<1$, that is,
a piece of the orbit in a 4-recurrent cylinder. Since by (\ref{eq11}) and (\ref{eq14})
\begin{align*}
\eta_1(\delta) &\leq   \eta_1(x_D) = \frac{\sin(\theta_2)}{\sin(\theta_1)}\\
s_3-\eta_1(1+x_D-\delta) &\geq  s_3-\eta_1(x_D)= s_3- \frac{\sin(\theta_2)}{\sin(\theta_1)}
\end{align*}
we conclude from Proposition \ref{prop2} that for $0\leq t \leq 4$ we have
$x_t^{[k_t]}\leq \sin(\theta_2)/\sin(\theta_1)$ whenever $k_t=2$ and
$x_t^{[k_t]}\geq s_3-\sin(\theta_2)/\sin(\theta_1)$ whenever $k_t=3$.

Similarly, consider  $[k_0,\phi_0^{[k_0]} ,x_0^{[k_0]}]=[1,\phi_*,\delta]$ with $0<\delta<x_D$,
that is, a part of the orbit in a 10-recurrent cylinder. Then by (\ref{eq8}) and (\ref{eq9})
\begin{align*}
\xi_1(\delta) &\leq  \xi_1(0) = \frac{\sin(\psi_0)}{\sin(\psi_1)}\\
\xi_2(\delta) &\geq  \xi_2(0)= s_3-\frac{\sin(\psi_5)}{\sin(\psi_2)}\\
\xi_4(\delta) &\geq  \xi_4(0) = s_3 -\frac{\sin(\psi_5)}{\sin(\psi_4)}\\
s_2-\xi_4(x_D-\delta) &\leq  s_2-\xi_4(0) = \frac{\sin(\psi_5)}{\sin(\psi_4)}\\
s_2 - \xi_2(x_D-\delta) &\leq  s_2-\xi_2(0)= \frac{\sin(\psi_5)}{\sin(\psi_2)}\\
s_3-\xi_1(x_D-\delta) &\geq  s_3-\xi_1(0)=s_3-\frac{\sin(\psi_0)}{\sin(\psi_1)} \, .
\end{align*}
Hence we conclude from Proposition \ref{prop1} that for $0\leq t \leq 10$ we have
$x_t^{[k_t]}\leq \varepsilon$ whenever $k_t=2$ and
$x_t^{[k_t]}\geq s_3-\varepsilon$ whenever $k_t=3$, where
\[ \varepsilon=\min\{s_2-\sin(\psi_0)/\sin(\psi_1), s_2-\sin(\psi_5)/\sin(\psi_2),
s_2-\sin(\psi_5/\sin(\psi_4)\}\,.\]

Altogether, the claim of the lemma is valid with the choice
\begin{displaymath}
\varepsilon=
\min\left \{s_2-\frac{\sin(\theta_2)}{\sin(\theta_1)},
s_2-\frac{\sin(\psi_0)}{\sin(\psi_1)}, s_2-\frac{\sin(\psi_5)}{\sin(\psi_2)},
s_2-\frac{\sin(\psi_5)}{\sin(\psi_4)}\right \}\,.
\end{displaymath}
Lemma \ref{lem5} and \ref{lem10} ensure that $\varepsilon>0$.
\end{proof}

Since $2\cos(\alpha)\cos(4 \alpha)=\cos(5 \alpha)-\cos(3 \alpha)$ it follows that choosing
$\alpha \in (\alpha_*,3\pi/10)$ such that $\cos(5 \alpha)/\cos(3 \alpha) \in
\mathbb{R}\setminus \mathbb{Q}$ ensures that the
map in Theorem \ref{thrm1} is an irrational rotation. In fact, the ratio $\cos(5 \alpha)/\cos(3 \alpha)$ 
is continuous and strictly monotonic for $\alpha \in (\alpha_*,3\pi/10)$, so that apart from a countable set
of $\alpha$ values we obtain an irrational rotation. For explicit examples
of such angles we invoke the Gelfond--Schneider Theorem 
(see, for example, \cite[Theorem~5.1]{BuTu_2004}), according to which for any algebraic $a\in \mathbb{R}\setminus\{0,1\}$ and algebraic irrational $b$, the number $a^b$ is transcendental. Thus, if $\alpha=\pi \beta$ with $\beta\neq 0$ an algebraic irrational number, then
$\cos(5 \alpha)/\cos(3 \alpha)$ must be irrational, otherwise $\exp(\mathrm{i}\alpha)=\exp(\mathrm{i}\pi\beta)=(-1)^\beta$ 
would be algebraic, which contradicts the Gelfond-Schneider Theorem. 

Hence, if  $\alpha \in (\alpha_*,3\pi/10)$ with  $\cos(5 \alpha)/\cos(3 \alpha) \in
\mathbb{R}\setminus \mathbb{Q}$, it follows that Lebesgue almost all initial values $x_0^{[1]}$ will give rise to a regular
non-periodic orbit with initial condition $[1,\phi_*,x_0^{[1]}]$. By
Lemma \ref{lem11} the corresponding trajectory does not have bounces
on the sides within a distance $\epsilon>0$ of the tip of the triangle (when distance is
measured along the bouncing side). Hence, the trajectory does not enter a small symmetric
triangular region at the tip of the triangle and is thus not everywhere dense. A graphical illustration of this type of trajectory
is shown in Figure~\ref{fig2}.

\begin{figure}[!h]
\centering
\includegraphics[width=0.5\textwidth]{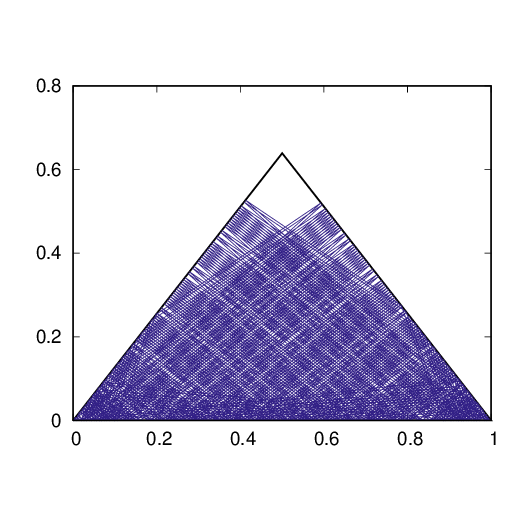}
\caption{Finite trajectory of 500 bounces with initial condition
on the base at $x_0^{[1]}=1/\sqrt{2}$, $\phi_0=0.7329252\ldots$, see (\ref{eq3}) in an isosceles triangle with base angle
$\alpha=\pi \sqrt{3}/6$ (see Lemma \ref{lem1}).}
\label{fig2}
\end{figure}

\section*{Acknowledgement}
The authors thank an anonymous  reviewer
for many constructive and insightful
suggestions, which greatly improved
the exposition of this article.

\section*{Funding and Competing interests}
All authors gratefully acknowledge the support for the research presented 
in this article by the EPSRC grant EP/RO12008/1. 
J.S. was partly supported by the ERC-Advanced Grant
833802-Resonances and 
W.J. acknowledges funding by the Deutsche Forschungsgemeinschaft
(DFG, German Research Foundation) SFB 1270/2 - 299150580.

The authors have no competing interests to declare that are relevant 
to the content of this article.


\begin{thebibliography}{10}

\bibitem{ArCaGu_PRE97}
R.~Artuso, G.~Casati, and I.~Guarneri.
\newblock Numerical study on ergodic properties of triangular billiards.
\newblock {\em Phys.\ Rev.\ E}, 55:6384, 1997.

\bibitem{BoTr_FM11}
J.~Bobok and S.~Troubetzkoy.
\newblock Does a billiard orbit determine its (polygonal) table?
\newblock {\em Fundamenta Mathematicae}, 212:129, 2011.

\bibitem{BuTu_2004}
E.B.~Burger and A.~Tubbs. 
\newblock {\em Making Transcendence Apparent}. 
\newblock New York, Springer, 2004. 

\bibitem{CaPr_PRL99}
G.~Casati and T.~Prosen.
\newblock Mixing property of triangular billiards.
\newblock {\em Phys.\ Rev.\ Lett.}, 83:4729, 1999.

\bibitem{ChMa:06}
N.~Chernov and R.~Markarian.
\newblock {\em Chaotic Billiards}.
\newblock Mathematical Surveys and Monographs, vol. 127, American Mathematical
  Society, 2006.

\bibitem{DaDiRuLa_AG18}
Diana Davis, Kelsey DiPietro, Jenny Rustad, and Alexander~{St} Laurent.
\newblock Negative refraction and tiling billiards.
\newblock {\em Adv. Geom.}, 18:133, 2018.

\bibitem{Galp_CMP83}
G.A. Galperin.
\newblock Non-periodic and not everywhere dense billiard trajectories in convex
  polygons and polyhedrons.
\newblock {\em Comm. Math. Phys.}, 91:187, 1983.

\bibitem{Gutk_ETDS84}
E.~Gutkin.
\newblock Billiards on almost integrable polyhedral surfaces.
\newblock {\em Ergod.\ Th.\ \& Dynam.\ Sys.}, 4:569, 1984.

\bibitem{Gutk_JSP96}
E.~Gutkin.
\newblock Billiards in polygons: Survey of recent results.
\newblock {\em J.\ Stat.\ Phys.}, 83:7, 1996.

\bibitem{Gutk_CHAOS12}
E.~Gutkin.
\newblock Billiard dynamics: An updated survey with the emphasis on open
  problems.
\newblock {\em Chaos}, 22:026116, 2012.

\bibitem{KeSm_CMV00}
R.~Kenyon and J.~Smillie.
\newblock Billiards on rational-angled triangles.
\newblock {\em Comment. Math. Helv.}, 75:65, 2000.

\bibitem{KeMaSm_AM86}
S.~Kerckhoff, H.~Masur, and J.~Smillie.
\newblock Ergodicity of billiard flows and quadratic differentials.
\newblock {\em Ann.\ Math.}, 115:293, 1986.

\bibitem{MaTa:02}
H.~Masur and S.~Tabachnikov.
\newblock Rational billiards and flat structures.
\newblock In B.~Hasselblatt and A.~Katok, editors, {\em Handbook of Dynamical
  Systems, Vol 1A}, page 1015. North Holland, Amsterdam, 2002.

\bibitem{Mcmu_IM06}
C.T. McMullen.
\newblock Teichm\"uller curves in genus two: torsion divisors and ratios of
  sines.
\newblock {\em Inv. Math.}, 165:651, 2006.

\bibitem{Toka_CMP15}
G.W. Tokarsky.
\newblock Galperin's triangle example.
\newblock {\em Comm. Math. Phys.}, 335:1211, 2015.

\bibitem{Trou_RCD04}
S.~Troubetzkoy.
\newblock Recurrence and periodic billiard orbits in polygons.
\newblock {\em Regul. Chaotic Dyn.}, 9:1, 2004.

\bibitem{Trou_AIF05}
S.~Troubezkoy.
\newblock Periodic billiard orbits in right triangles.
\newblock {\em Ann. Inst. Four.}, 55:29, 2005.

\bibitem{Veec_IM89}
W.A. Veech.
\newblock Teichm\"uller curves in moduli space, Eisenstein series and an
  application to triangular billiards.
\newblock {\em Inv. Math.}, 97:553, 1989.

\bibitem{WaCaPr_PRE14}
J.~Wang, G.~Casati, and T.~Prosen.
\newblock Nonergodicity and localization of invariant measure for two colliding
  masses.
\newblock {\em Phys.\ Rev.\ E}, 89:042918, 2014.

\bibitem{ZaSlBaJu_PRE22}
K.~Zahradova, J.~Slipantschuk, O.~F. Bandtlow, and W.~Just.
\newblock Impact of symmetry on ergodic properties of triangular billiards.
\newblock {\em Phys. Rev.\ E}, 105:L012201, 2022.

\bibitem{ZaSlBaJu_PD23}
K.~Zahradova, J.~Slipantschuk, O.F. Bandtlow, and W.~Just.
\newblock Anomalous dynamics in symmetric triangular irrational billiards.
\newblock {\em Physica D}, 445:133619, 2023.

\bibitem{ZeKa_MNAS75}
A.~N. Zemlyakov and A.~B. Katok.
\newblock Topological transitivity of billiards in polygons.
\newblock {\em Math. Notes Acad.\ Sci.\ USSR}, 18:760, 1975.

\end{thebibliography}
\end{document}